\documentclass[6pt]{article}
\usepackage[fleqn]{amsmath}
\usepackage{amssymb, amsthm, latexsym, amsfonts, epsfig,  graphicx,epstopdf,subcaption,caption,mathrsfs}
\usepackage{float}
\usepackage{ulem}
\usepackage{enumitem}
 \usepackage{color}
\numberwithin{equation}{section}

\def\kasten{$~~\mbox{\hfil\vrule height6pt width5pt depth-1pt}$ }

\usepackage{geometry}
\newtheorem{Theorem}{Theorem}[section]

\newtheorem{Remark}[Theorem]{Remark}

\begin{document}

\def\kasten{$~~\mbox{\hfil\vrule height6pt width5pt depth-1pt}$}
\makeatletter\def\theequation{\arabic{section}.\arabic{equation}}
\makeatother

\newtheorem{theorem}{Theorem}
\centerline{\large \bf A role of random slow manifolds in detecting stochastic bifurcation
\footnote{This work was partly supported by the
 NSFC grant 11531006, 11771449 and 11301197.}\hspace{2mm}
\vspace{1mm}\vspace{1mm}\\ }
\smallskip
\centerline{\bf Ziying He$^{a,b,c}\footnote{ziyinghe@hust.edu.cn}$,
\bf Rui Cai$^{a,b,c}$\footnote{cairui@hust.edu.cn} 
\bf Jinqiao Duan$^{a,b,c,d}\footnote{duan@iit.edu}$
\bf Xianming Liu$^{b,c}\footnote{xmliu@hust.edu.cn}$}
\smallskip
\centerline{${}^a$ Center for Mathematical Sciences,}
 \centerline{${}^b$School of Mathematics and Statistics,}
 \centerline{${}^c$ Hubei Key Laboratory for Engineering Modeling and Scientific Computing}
\centerline{Huazhong University of Sciences and Technology, Wuhan 430074,  China}
\centerline{${}^d$Department of Applied Mathematics, Illinois Institute of Technology, Chicago, IL 60616, USA}

\begin{abstract}
We consider the relation for the stochastic equilibrium states between the reduced system on a random slow manifold and the original system. This provides a theoretical basis for the reduction about sophisticated detailed models by the random slow manifold without significant damage to the overall qualitative properties. Based on this result, we reveal a role of random slow manifolds in detecting  stochastic bifurcation by an example. The example exhibits a stochastic bifurcation phenomenon and possesses a random slow manifold that carries the stochastic bifurcation information of the system. Specifically, the lower dimensional reduced system on the random slow manifold retains the stochastic bifurcation phenomenon from the original system.
\end{abstract}

{\bf Key Words and Phrases}: Random slow manifolds,  stochastic bifurcation, multi-scale system, stochastic equilibrium state, singular perturbation method.

{\footnotesize \textbf{2010 Mathematics Subject Classification}:  60H10, 58J65.}
\section{Introduction}
An enormous growth of a system in both dimension and complexity aspects to develop more sophisticated detailed models to increase efficiency and to advance the design of industrial, engineering and technological tools (such as internal combustion engines, gas turbines, power plants etc.) as well as methods of organic and inorganic synthesis (biotechnology, structural molecular design etc.) \cite{Bykov}. 
The co-existence of disparate time scales is pervasive in many systems. In \cite{Nascimento}, the authors present a numerical investigation of the coupled slow and fast surface dynamics in an electrocatalytic oscillator. It is of a great interest to have a suitable tool for model reduction for slow-fast system without significant damage to the overall qualitative and quantitative properties of the reduced model by comparison with the original one \cite{Bykov}. Slow reduction method is one of such effective tool. In \cite{Floyd}, the authors consider a low-dimensional manifold reduction schemes for actin polymerization dynamics, that is a quasi-steady-state approximation scheme that leverages fast dynamics of the filament tips leading to a five-dimensional system for the 11-dimensional Brooks-Carlsson model. 

We consider a random slow reduction for the stochastic dynamical systems by the random slow manifolds in this paper. Random slow manifolds are geometrical invariant structures  of multi-scale stochastic dynamical systems. It has been shown that the deterministic slow manifold has dramatic effect on the overall dynamical behavior. In \cite{Grafke}, the authors discover that the bifurcation structure of the deterministic slow manifold creates a reaction channel for non-equilibrium transitions leading to vastly increased transition rates. We find that the random slow manifold also has a profound impact on the stochastic bifurcation for the stochastic dynamical system. The existence and the exponential tracking property of random slow manifolds are known \cite{Schmalfu-Schneider, Hongbo Fu}. It is beneficial to take advantage of random slow manifolds in order to examine
dynamical behaviors of   multi-scale stochastic systems, either through the slow manifolds themselves or by the reduced systems on these slow manifolds \cite{Berglund}.   In fact, the random slow manifolds are utilized to describe  the settling of the inertial particle under uncertainty \cite{Jian Ren},     estimate a system parameter \cite{Jian Ren2}, and  understand  certain chemical reactions \cite{Xiaoying Han}.
Todd L Parsons and Tim Rogers derive explicit and general formulae for a reduced-dimension description for well-separated slow and fast stochastic dynamics in the limit of small noise and explore the Michaelis-Menten law of enzyme-catalysed reactions, and the link between the Lotka-Volterra and Wright-Fisher processes as specific examples in \cite{Parsons}. 

In this present paper, we consider the stochastic bifurcation phenomenon with help of a random slow manifold. The related bifurcation phenomenon was also considered in \cite{George W A Constable} through an ecological model of two interacting populations and a simplified epidemiological model, which only involves the deterministic slow manifold. First, we prove the main result  about the stochastic equilibrium states for the original system and the reduced system on the random slow manifold.  Second, according to the main result, we reveal the relation for the stochastic bifurcation phenomenon between the original system and the reduced system on the random slow manifold by an example. The reduced system bears the stochastic bifurcation phenomenon inherited from the original system. Thus we could detect stochastic bifurcation by examining the low dimensional slow system.  By a singular perturbation method \cite{Robert E.,Jian Ren} to derive an explicit but approximate   random slow manifold, we obtain a lower dimensional reduced system on the random slow manifold. Then we illustrate that the slow system captures the stochastic bifurcation of the original system through numerical simulations.

\section{Random slow manifolds and stochastic bifurcation}\label{sectionbifurcation}

We consider the following multi-scale stochastic dynamical system in $\mathbb{R}^{d}$,
\begin{eqnarray}\label{orig}
\left\{\begin{array}{l}
\mathrm{d}x=\varepsilon(Ax+f(x,y))\mathrm{d}t,\qquad\qquad\,\,\,\,\, \mbox{in } \mathbb{R}^{m},\\
\mathrm{d}y=(By+g(x,y))\mathrm{d}t+\sigma\mathrm{d}B_{t},\qquad \mbox{in } \mathbb{R}^{n}.
\end{array}
\right.
\end{eqnarray}
with $m+n=d$. The Euclidean metric in $\mathbb{R}^{d}$, $\mathbb{R}^{m}$ and $\mathbb{R}^{n}$ are denoted by $d(\cdot,\cdot)$, $d_{1}(\cdot,\cdot)$ and $d_{2}(\cdot,\cdot)$ respectively, and $d=d_{1}+d_{2}$. The parameter $\varepsilon\ll1$ quantifies  the time scales for the slow $x$ and fast $y$ components. Here $A, B $ are matrices, $f, g$ are nonlinear vector functions, and $\sigma$ is noise intensity. Brownian motion $B_{t}$ is defined on a sample space $(\Omega,\mathbb{P},\mathcal{F})$; refer to \cite{Ludwig Arnold,JinqiaoDuan}.

For a stochastic dynamical system, a stochastic equilibrium state is a state which attracts or repels all nearby solutions \cite{Schmalfuss}. It is called a stable stochastic equilibrium state if it attracts all nearby solutions. It is called an unstable stochastic equilibrium state, if it repels all nearby solutions.
Assume that  the system \eqref{orig} satisfies the conditions in \cite{Hongbo Fu}, that is exponential dichotomy, Lipschitz condition of the nonlinear term and the gap condition between the spectral and the Lipschitz constant. Then the main result in \cite{Hongbo Fu} holds for the system \eqref{orig}, i.e., there exists a random slow manifold with exponential tracking property,
\begin{eqnarray}\label{smf}
M(\omega)=\{(\xi,h(\xi,\omega)):\xi\in\mathbb{R}^m, \omega\in\Omega\},
\end{eqnarray}
through which the original system  \eqref{orig} can be reduced to a lower dimensional system,
\begin{eqnarray}\label{reduce}
\mathrm{d}x=\varepsilon(Ax+f(x,h(x,\omega)))\mathrm{d}t, \qquad\mbox{in } \mathbb{R}^{m}.
\end{eqnarray}
Furthermore, the reduced system can be used to study the dynamical structure of the original system. In fact, the following result holds. A relation about the random dynamical structure of the original system and the reduced system on a  random \emph{center} manifold was considered in \cite{Boxler,Chen}.

\begin{theorem}\label{equivalent}
Let $h$ be the representation of the random slow manifold in  (\ref{smf}).
If $x^{\star}(\omega)$ is a stable (unstable) stochastic equilibrium state of the reduced system \eqref{reduce}, then $(x^{\star}(\omega),h(x^{\star}(\omega),\omega))$ is also a stable (unstable) stochastic equilibrium state for the original system \eqref{orig}.
\end{theorem}
\begin{proof}
For the random dynamical system $\phi$ generated by the reduced system \eqref{reduce},
$$\phi(t,\omega,x^{\star}(\omega))=x^{\star}(\theta_{t}\omega).$$
Let $\Phi$  be  the random dynamical system generated by the original system \eqref{orig}. Due to the invariance of the ransom slow manifold,
$$\Phi(t,\omega,M(\omega))\subset M(\theta_{t}\omega).$$
Note that $x$ is a solution of the reduced system \eqref{reduce}  if and only if $(x,h(x,\theta_{t}\omega)))$ is a solution of the original system \eqref{orig}. We have
\begin{eqnarray*}
\begin{aligned}
&\Phi(x^{\star}(\omega),h(x^{\star}(\omega)))\\
=&\bigg(x\big(t,\omega,(x^{\star}(\omega),h(x^{\star}(\omega),\omega))\big),y\big(t,\omega,(x^{\star}(\omega),h(x^{\star}(\omega),\omega))\big)\bigg)\\
=&\bigg(x\big(t,\omega,(x^{\star}(\omega),h(x^{\star}(\omega),\omega))\big),h\big(x(t,\omega,(x^{\star}(\omega),h(x^{\star}(\omega),\omega)),\theta_{t}\omega)\big)\bigg)\\
=&\big(x(t,\omega,x^{\star}(\omega)),h(x(t,\omega,x^{\star}(\omega)),\theta_{t}\omega)\big)\\
=&(\phi(t,\omega,x^{\star}(\omega)),h(\phi(t,\omega,x^{\star}(\omega)),\theta_{t}\omega))\\
=&(x^{\star}(\theta_{t}\omega),h(x^{\star}(\theta_{t}\omega),\theta_{t}\omega))\\
\end{aligned}
\end{eqnarray*}
The first equality is the property of the flow $\Phi$. The second equality is due to the invariance of the random slow manifold. The third equality holds by the correspondence relation for solutions of the original system  \eqref{orig} and the reduced system \eqref{reduce}. The fourth equality is the property of the flow $\phi$. The last equality holds by the condition that $x^{\star}(\omega)$ is a stochastic equilibrium state of the reduced system \eqref{reduce}.

Assume that  $x^{\star}(\omega)$ is a stochastic equilibrium state  of the reduced system \eqref{reduce}. Specifically, for any state $x_{0}(\theta_{-t}\omega)$ in some  neighbourhood $U(x^{\star}(\theta_{-t}\omega);\delta)$,
$$d_{1}(x(t,\theta_{-t}\omega,x_{0}(\theta_{-t}\omega)),x^{\star}(\omega))\to0, \mbox{ as } t\to 0 .$$
Then for any state $(x_{0}(\theta_{-t}\omega),y_{0}(\theta_{-t}\omega))$ in the neighbourhood $U((x^{\star}(\theta_{-t}\omega),h(x^{\star}(\theta_{-t}\omega)));\delta)$ of the state $(x^{\star}(\theta_{-t}\omega),h(x^{\star}(\theta_{-t}\omega)))$,
\begin{eqnarray*}
\begin{aligned}
&d\bigg(\Phi\big(t,\theta_{-t}\omega,(x_{0}(\theta_{-t}\omega),y_{0}(\theta_{-t}\omega))\big),\big(x^{\star}(\omega),h(x^{\star}(\omega)\big)\bigg)\\
\leq&d\bigg(\Phi\big(t,\theta_{-t}\omega,(x_{0}(\theta_{-t}\omega),y_{0}(\theta_{-t}\omega))\big), \big(\phi(t,\theta_{-t}\omega,x_{0}(\theta_{-t}\omega)),h(\phi(t,\theta_{-t}\omega,(x_{0}(\theta_{-t}\omega))))\big)\bigg)\\
&+d\bigg(\big(\phi(t,\theta_{-t}\omega,x_{0}(\theta_{-t}\omega)),h(\phi(t,\theta_{-t}\omega,(x_{0}(\theta_{-t}\omega))))\big),\big(x^{\star}(\omega),h(x^{\star}(\omega)\big)\bigg)\\
\leq&M_{1}e^{-M_{2}t}d\big((x_{0}(\theta_{-t}\omega),y_{0}(\theta_{-t}\omega)),(x_{0}(\theta_{-t}\omega),h(x_{0}(\theta_{-t}\omega)))\big)+d_{1}\big(\phi(t,\theta_{-t}\omega,x_{0}(\theta_{-t}\omega)),x^{\star}(\omega)\big)\\
&+d_{2}\big(h(\phi(t,\theta_{-t}\omega,(x_{0}(\theta_{-t}\omega)))),h(x^{\star}(\omega))\big)\\
\leq&\delta M_{1}e^{-M_{2}t}+(1+L_{h})\cdot d_{1}\big(\phi(t,\theta_{-t}\omega,x_{0}(\theta_{-t}\omega)),x^{\star}(\omega)\big)\\
\to&0,  \quad\mbox{as}\quad t\to+\infty.
\end{aligned}
\end{eqnarray*}
In the second inequality, the positive numbers $M_{1}$ and $M_{2}$ exist,  due to the exponential attracting property about the reduced system to the original system. The third inequality holds with the Lipschitz constant $L_{h}$ of the random slow manifold. The last convergence relation is due to the fact that $x^{\star}(\omega)$ is a stable stochastic equilibrium state.

This result means $(x^{\star}(\omega),h(x^{\star}(\omega),\omega))$ attracts all the states in $U((x^{\star}(\theta_{-t}\omega),h(x^{\star}(\theta_{-t}\omega)));\delta)$, and is thus a stable   stochastic equilibrium state for the original system \eqref{orig}.  It is similar  to prove the case for a unstable stochastic equilibrium state.
\end{proof}
\begin{Remark}
If the stable stochastic equilibrium state $x^{\star}(\omega)$ exponentially attracts all the nearby solutions of the reduced system \eqref{reduce}, then the stable stochastic equilibrium state $(x^{\star}(\omega),h(x^{\star}(\omega),\omega))$ also attracts all the nearby solutions of the original system  \eqref{orig} exponentially. This can be seen from the proof   above.
\end{Remark}
This result can be used to detect stochastic bifurcation of the original system, by examining the number and stability type of stochastic equilibrium states of the lower dimensional slow system.

We now illustrate this by the following example, with a bifurcation parameter $a\in \mathbb{R}$,
\begin{eqnarray}\label{bfeq}
\left\{\begin{array}{l}
\mathrm{d}x=-\varepsilon(ax+\frac{y}{1+x^2})\mathrm{d}t,\qquad\qquad\,\,\,\, x(0)=\xi,\qquad\, \mbox{in } \mathbb{R},\\
\mathrm{d}y=(-2y-\sin x)\mathrm{d}t+\sigma\mathrm{d}B_{t},\qquad y(0)=\eta,\qquad \mbox{in } \mathbb{R}.
\end{array}
\right.
\end{eqnarray}
When parameter $a$ varies, the stochastic system \eqref{bfeq} undergoes an interesting  stochastic bifurcation phenomenon indicated by the stable equilibrium states.  We show this in the left hand side of Figure \ref{figbfreduc} with several solutions simulated by the stochastic implicit Euler scheme \cite{Peter Kloeden}. Figure \ref{figbfreduc} exhibits that the number of the stable equilibrium states for the system \eqref{bfeq} will vary with $a$. For a large positive value of $a$, the number equals to one. As  $a$ decreasing to zero, the number will increase. When $a=0$, there are multiple equilibrium states. For the negative value of $a$, the number will decrease with the decreasing  $a$. It has only one stable equilibrium state for $a=0.6$,  two for $a=0.01$, four for $a=0.001$,    six for $a=0$ (we only plot in a finite interval), four for $a=-0.0003$, two for $a=-0.001$ and zero for $a=-0.006$ in the figure.

This example satisfies the conditions in \cite{Hongbo Fu}, that is exponential dichotomy, Lipschitz condition of the nonlinear term and the gap condition between the spectral and the Lipschitz constant. Thus it has  a random slow manifold, and it can be reduced to a one dimensional system on the random slow manifold. The reduced system also has a stochastic bifurcation phenomenon exhibited by the change of the number of   stochastic equilibrium states. According to   Theorem \ref{equivalent}, the number of its stochastic equilibrium state of the reduced one dimensional system is the same with that of the system \eqref{bfeq}. We will show this   as follows.

We first derive the analytic approximate expression of the random slow manifolds. Introduce $\mathrm{d}z=-2z+\sigma\mathrm{d}B_{t}.$ It has  a stationary solution $z(t)=\sigma\int_{-\infty}^{t}e^{-2(t-s)}\mathrm{d}B_{s}.$
By the transformation
\begin{eqnarray}\label{transform}
T(x,y-z)=(u,v),
\end{eqnarray}
the system \eqref{bfeq} is transformed to
\begin{eqnarray}\label{bfeqr}
\left\{\begin{array}{l}
\dot{u}=-\varepsilon(au+\frac{v+z}{1+u^2}), \qquad u(0)=\xi,\qquad\qquad\quad \mbox{in } \mathbb{R},\\
\dot{v}=-2v-\sin u, \qquad\quad\,\, v(0)=\eta-z(0),\qquad \mbox{in } \mathbb{R}.
\end{array}
\right.
\end{eqnarray}
The original system \eqref{bfeq} and the transformed system \eqref{bfeqr} are conjugated through the transformation \eqref{transform}. \\
We can infer the Lyapunov-Perron equation \cite{Hongbo Fu},
\begin{eqnarray*}
\left\{\begin{array}{l}
u(t)=e^{-\varepsilon a t}u(0)+\int_{0}^{t}e^{-\varepsilon a(t-s)}\frac{v+z}{1+u^2}\mathrm{d}s\\
v(t)=-\int_{-\infty}^{t}e^{-2(t-s)}\sin u \mathrm{d}s.
\end{array}
\right.
\end{eqnarray*}
According to this equation and the Lyapunov-Perron method \cite{Hongbo Fu}, the random slow manifold of the system \eqref{bfeqr}   is a graph
\begin{eqnarray}\label{smf}
\widetilde{M}(\omega)=\{(\xi,\tilde{h}(\xi,\omega)):\xi\in\mathbb{R}, \omega\in\Omega\},
\end{eqnarray}
with an implicit expression $\tilde{h}(\xi,\omega)=-\int_{-\infty}^{0}e^{2s}\sin u\mathrm{d}s.$
Due to the conjugation relation between the two systems, the random slow manifolds of the original system \eqref{bfeq} exists and is a graph
\begin{eqnarray}\label{smf}
M(\omega)=\{(\xi,h(\xi,\omega)):\xi\in\mathbb{R}, \omega\in\Omega\},
\end{eqnarray}
with an implicit expression $h(\xi,\omega)=z(0)-\int_{-\infty}^{0}e^{2s}\sin x\mathrm{d}s.$  Let us find an approximate  expression for the graph  $h$.\\
By the singular perturbation theory \cite{Robert E., Jian Ren}, we can obtain an explicit expression,
\begin{eqnarray}\label{bifurexplicitexpression}
\begin{aligned}
h(\xi,\omega)=z_{0}-\frac{\sin \xi}{2}+\varepsilon\bigg(-\frac{a \xi\cos \xi}{4}+\frac{\sin \xi\cos \xi}{8(1+\xi^2)}+\frac{\sigma\cos \xi}{2(1+\xi^2)}\int_{-\infty}^{0}\tau e^{2\tau}\mathrm{d}B_{\tau}(\omega)\bigg)+o(\varepsilon^2).
\end{aligned}
\end{eqnarray}
Thus, we can reduce the original system \eqref{bfeq} to the following approximate slow system,   omitting the higher order of $o(\varepsilon^2)$  
\begin{eqnarray}\label{bfred}
\begin{aligned}
\dot{x}=-\varepsilon\bigg\{ax+\frac{1}{1+x^2}\bigg[z-\frac{\sin x}{2}-\varepsilon\bigg(\frac{ax\cos x}{4}-\frac{\sin x\cos x}{8(1+x^2)}-\frac{\sigma\cos x}{2(1+x^2)}\int_{-\infty}^{t}(\tau-t)e^{2(\tau-t)}\mathrm{d}B_{\tau}(\omega)\bigg)\bigg]\bigg\}.
\end{aligned}
\end{eqnarray}
We reveal a  stochastic bifurcation phenomenon about this reduced system in the right hand side of Figure \ref{figbfreduc} with several solutions simulated by the stochastic implicit Euler scheme \cite{Peter Kloeden}.
\begin{figure}[H]
\centering
\includegraphics[width=3.5in]{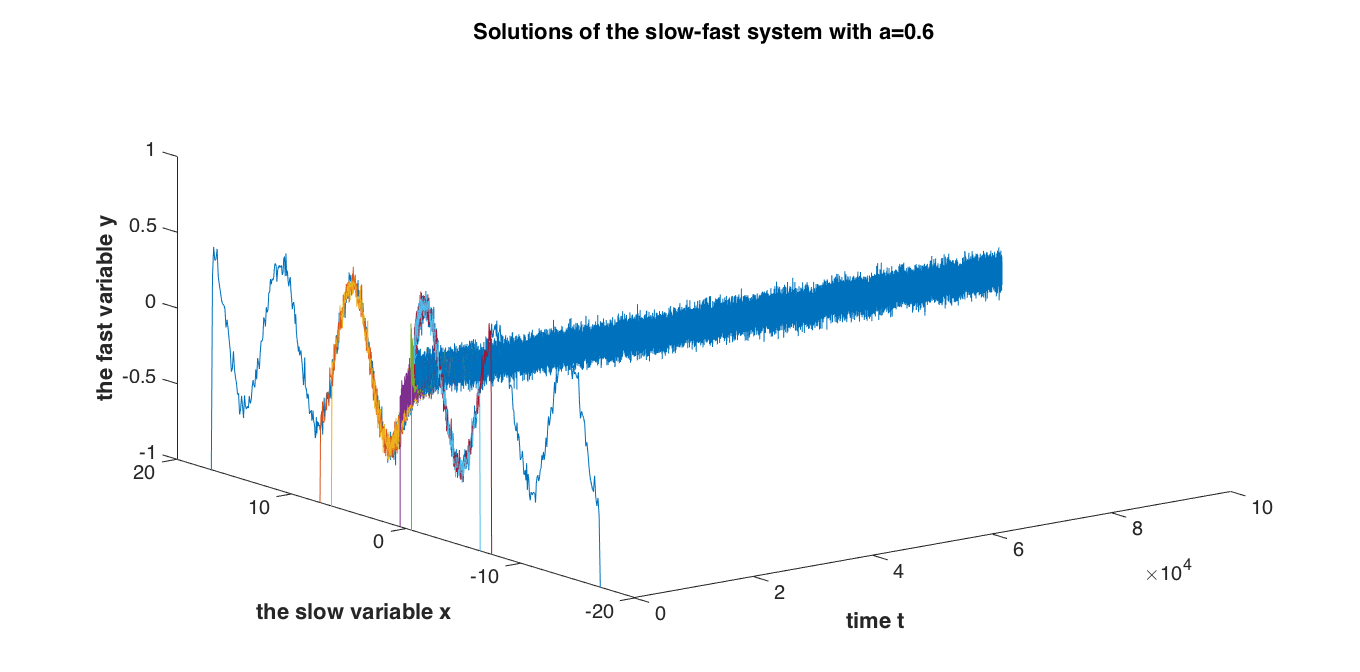}
\includegraphics[width=2in]{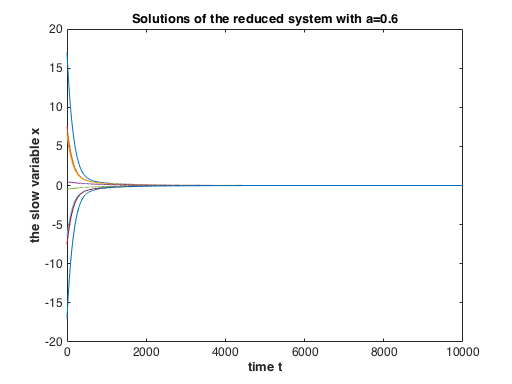}
\end{figure}
\begin{figure}[H]
\centering
\includegraphics[width=3.5in]{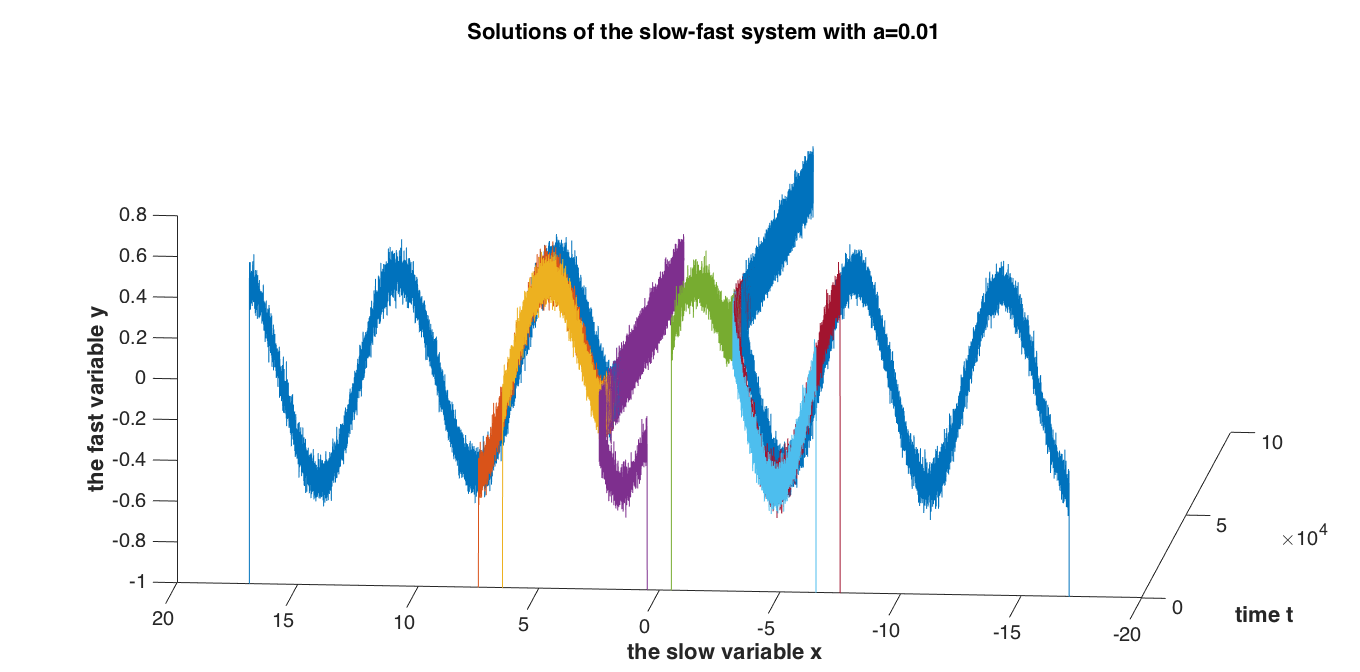}
\includegraphics[width=2in]{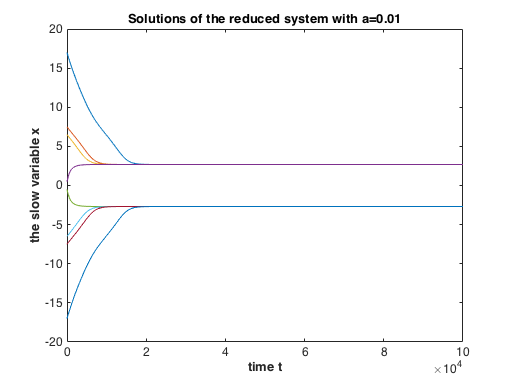}
\end{figure}
\begin{figure}[H]
\centering
\includegraphics[width=3.5in]{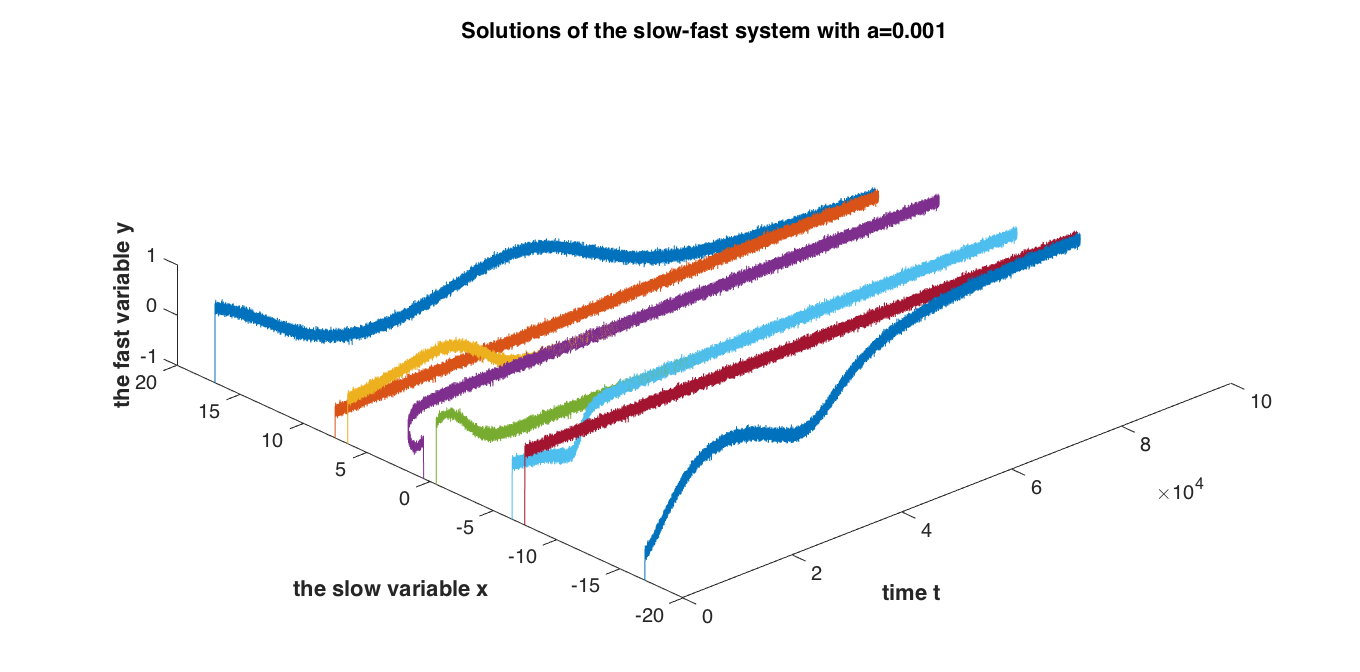}
\includegraphics[width=2in]{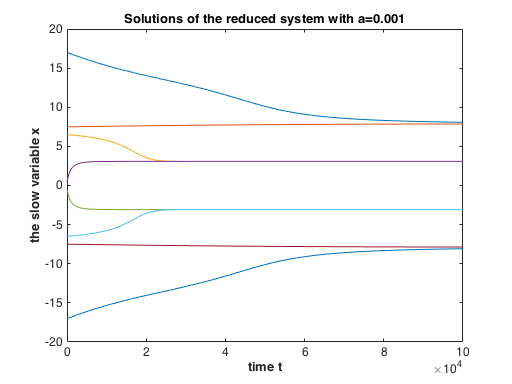}
\end{figure}
\begin{figure}[H]
\centering
\includegraphics[width=3.5in]{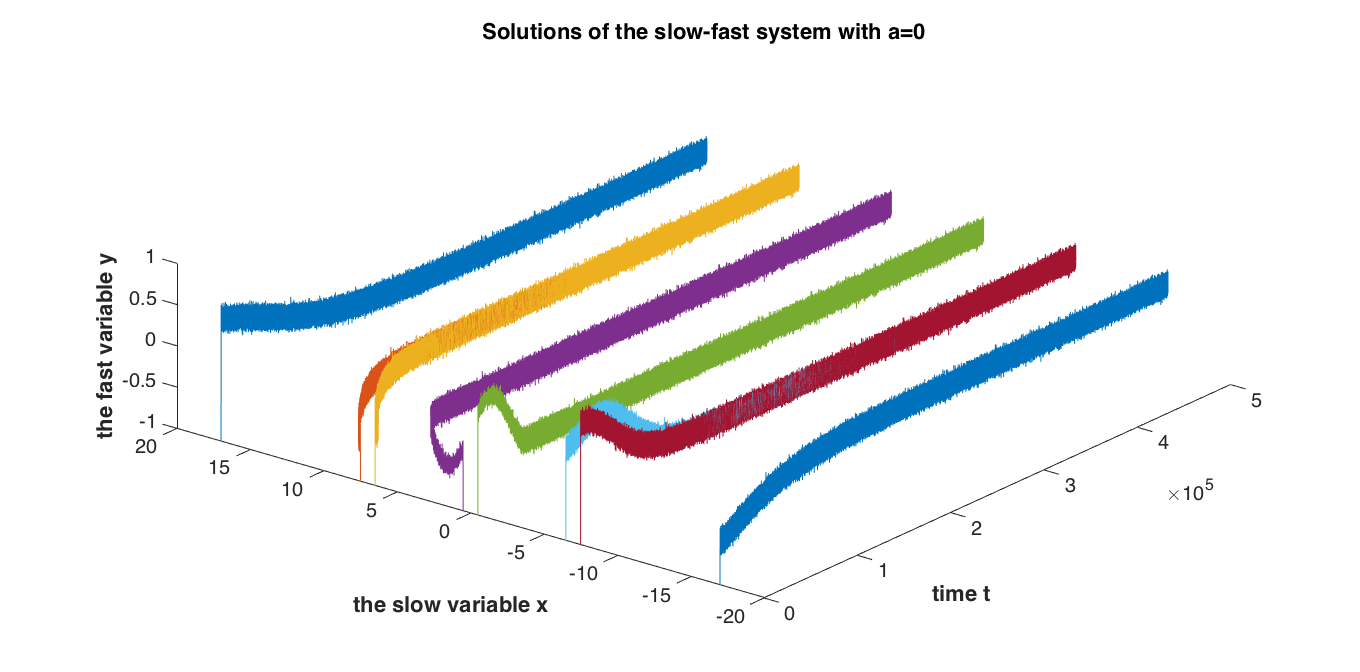}
\includegraphics[width=2in]{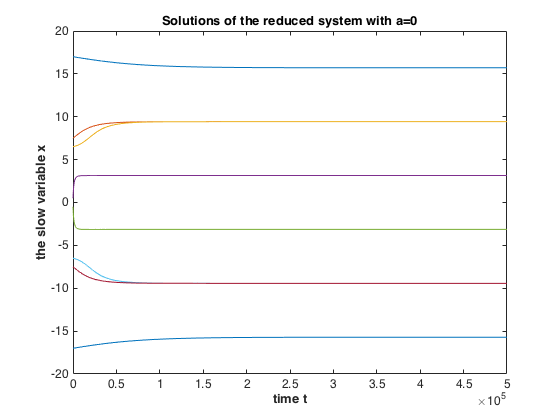}
\end{figure}
\begin{figure}[H]
\centering
\includegraphics[width=3.5in]{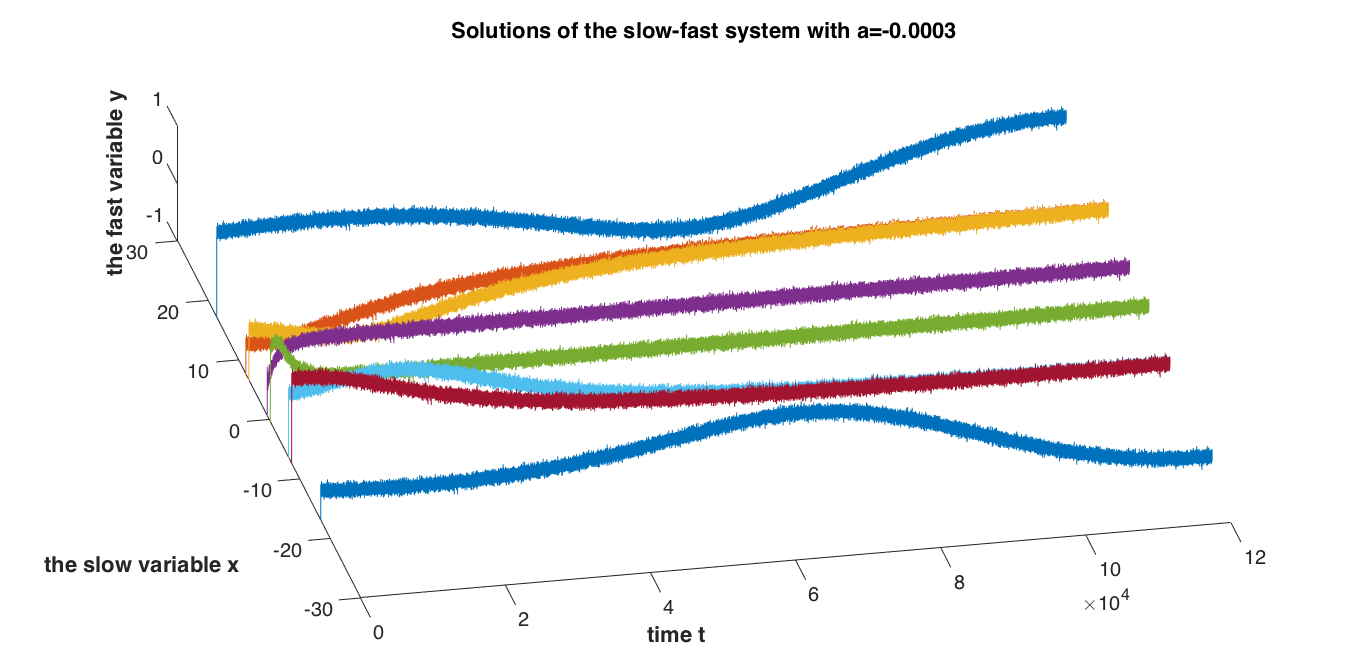}
\includegraphics[width=2in]{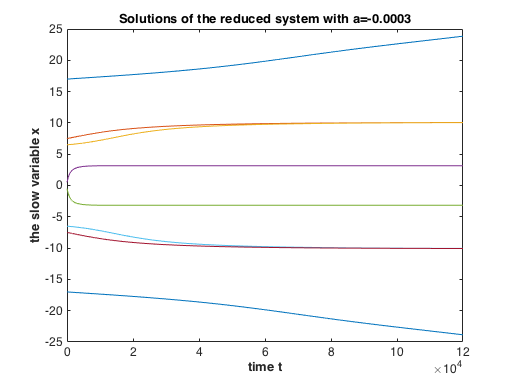}
\end{figure}
\begin{figure}[H]
\centering
\includegraphics[width=3.5in]{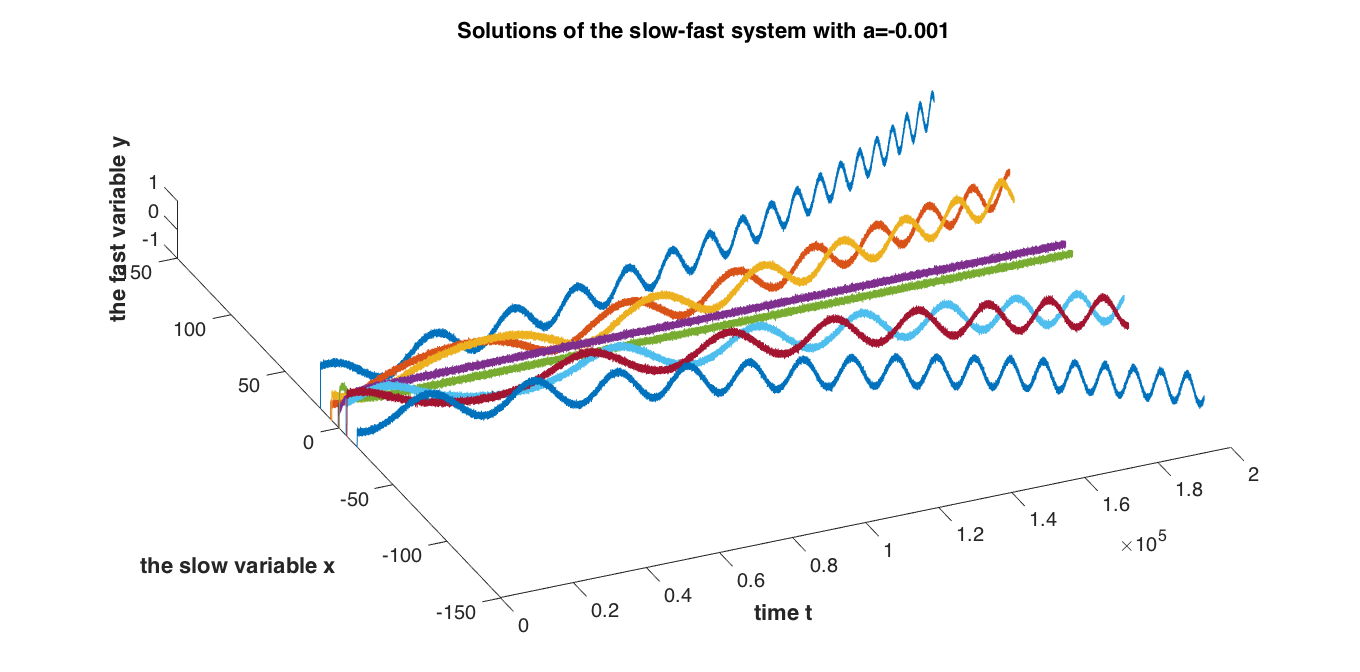}
\includegraphics[width=2in]{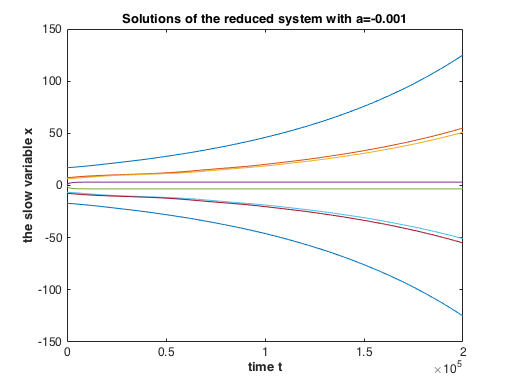}
\end{figure}
\begin{figure}[H]
\centering
\includegraphics[width=3.5in]{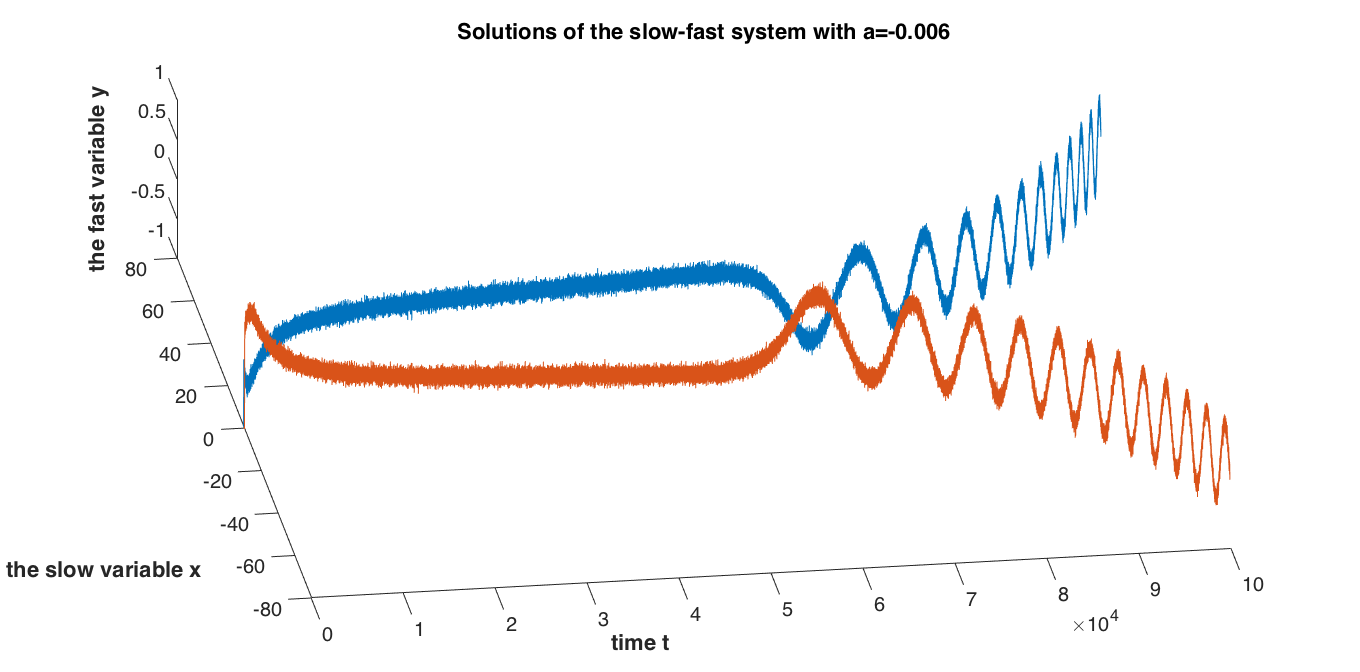}
\includegraphics[width=2in]{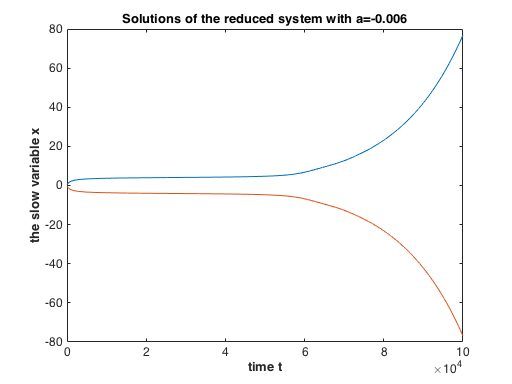}
\end{figure}
\begin{figure}[H]
\centering
\caption{(online color) Several solutions of the system \eqref{bfeq} in the left hand side and the reduced system \eqref{bfred} in the right hand side for parameter $a=0.6$, $0.01$, $0.001$, $0$, $-0.0003$ and $-0.001$ with initial value $(17,-1)$, $(7.5,-1)$, $(6.5,-1)$, $(0.5,-1)$, $(-0.5,-1)$, $(-6.5,-1)$, $(-7.5,-1)$, $(-17,-1)$ respectively. The last picture is two solutions for parameter $a=-0.006$ with initial value $(0.5,-1)$ and $(-0.5,-1)$. The scale parameter $\varepsilon=0.01$ and the noise intensity $\sigma=0.1$.}\label{figbfreduc}
\end{figure}
Figure \ref{figbfreduc} shows that the reduced system \eqref{bfred} exhibits a stochastic bifurcation phenomenon consistent  with the original system \eqref{bfeq}. The number of the stable equilibrium states are the same for   both  systems with the same value of $a$. The position of the stable equilibrium states of the reduced one dimensional system \eqref{bfred}  is also corresponding to that of the original     system \eqref{bfeq}. This simulation result confirms   Theorem \ref{equivalent} in this concrete example. And this means that the reduced system captures  the stochastic bifurcation information for the original system, and thus simplifies the bifurcation analysis of a higher dimensional system.
\section{Conclusions}
The random slow manifold with exponential tracking property plays a significant role on the dynamical behaviors of a slow-fast stochastic dynamical system. To get the slow manifold, we employ the approximation    method via perturbation.    The reduced system on the random slow manifold inherits the stochastic equilibrium states from the original system and thus may be used to detect stochastic bifurcation for the original higher dimensional system, we find that it is closed related to the stochastic bifurcation phenomenon. We have illustrated this bifurcation result with a simple example.

\end{document}